\documentclass[11pt,psfig]{article}
\usepackage{amsfonts,amssymb,amstext}
\usepackage{amsmath,amsthm}
\usepackage{graphicx}
\usepackage{url}
\usepackage{hyperref}

\def\R{\mathbf{R}}
\def\mixbound{O(n^{2/d} (\log n)^{(d-2)/d})}
\def\omn{\kappa(n)} 
\def\const{c'}

\def\al{\alpha}
\def\dl{\delta}

\def\ep{\epsilon}
\def\eps{\epsilon}
\def\th{\theta}

\def\vol{\mathrm{vol}}

\def\pr{\mathbb{P}}
\def\Pr{\mathbb{P}}
\def\E{\mathbb{E}}

\def\Bin{\textrm{Bin}}

\def\rc{r_{\mathrm{conn}}}

\newtheorem{theorem}{Theorem}

\newtheorem{lemma}[theorem]{Lemma}

\newenvironment{proofof}[1]{\vspace{1ex}\noindent{\bf Proof of #1.}\hspace{0.5em}}
	{\hfill\qed\vspace{1ex}}

\title{On the Mixing Time of Geographical Threshold Graphs}
\author{Andrew Beveridge\thanks{Department of Mathematics, Statistics and Computer Science, Macalester College, Saint Paul, MN 55105, USA. E-mail: \texttt{abeverid@macalester.edu}. 
}, 
Milan Bradonji\'c\thanks{Mathematics of Networks and Communications, Bell Laboratories, Alcatel-Lucent, 600 Mountain Avenue, Murray Hill, NJ 07974, USA. E-mail: \texttt{milan@research.bell-labs.com}. 
This work was done while the author was at Los Alamos National Laboratory.
}
}

\date{}

\begin{document}
\maketitle

\begin{abstract}
We study the mixing time of random graphs in the $d$-dimensional 
toric unit cube $[0,1]^d$ generated by the geographical 
threshold graph (GTG) model, a generalization of random geometric graphs (RGG). In a GTG, 
nodes are distributed in a Euclidean space, and edges are assigned according to a threshold 
function involving the distance between nodes as well as randomly chosen node weights, drawn from some distribution. 
The connectivity threshold for GTGs is comparable to that of RGGs, essentially corresponding to a connectivity radius of $r=(\log n/n)^{1/d}$. However, the degree distributions at this threshold are quite different: in an RGG the degrees are essentially uniform, while RGGs have heterogeneous degrees that depend upon the weight distribution.
Herein, we study the mixing times of random walks on $d$-dimensional GTGs near the connectivity threshold for $d \geq 2$. If the weight distribution function  decays with $\pr[W \geq x] = O(1/x^{d+\nu})$ for an arbitrarily small constant $\nu>0$ then the mixing time of GTG is $\mixbound$. This matches the known mixing bounds for the $d$-dimensional RGG.

\end{abstract}

\section{Introduction}
\label{sec:intro}

In recent years, we have witnessed the development of numerous approaches  to study the
structure of large real-world technological and social networks, and to
optimize processes on these networks.    
Large networks, such as the Internet, World Wide Web, phone call graphs, 
infections disease contacts and financial transactions, have provided
new challenges for modeling and analysis~\cite{Bonato-2005-Surveymodels}.
As an example, Web graphs may have billions of nodes and edges,
which implies that processing and
extracting information on these large sets of data, is 
`hard'~\cite{Abello-2002-MassiveDataSets}.
Extensive theoretical and experimental research has been done
in web-graph modeling, attempting to capture
both the structure and dynamics of the web 
graph~\cite{bonato-2009-waw,aiello-2007-waw,flaxman-2007-waw,flaxman-2007-intmath,kumar-2000-stochastic,barabasi-1999-emergence,aiello00random,bollobas-2001-degree,cooper01general}.

In general, a particularly fertile approach is to consider the network as an instance of an ensemble, arising from a suitable random generative model.       
Since the seminal papers on the evolution of uniform random graph model~\cite{erdos-1959-random,erdos-1960-evolution}, 
many other models have been
proposed to better capture the structure seen in real-world
networks, which are systematically covered in~\cite{durret-2005-random}.
One straightforward example is
the random geometric graph (RGG) model, where nodes are placed
uniformly at random in a Euclidean space and edges are placed between
any two nodes within a threshold distance. For further study of RGGs, see the monograph by Penrose~\cite{penrose:book}.
The RGGs have the advantage of
describing many aspects of systems such as sensor networks, while
avoiding unnecessary details. 
However, they fail to capture heterogeneity in the network.

Geographical threshold graphs (GTGs), introduced in \cite{masuda-2005-geographical}, are a generalization of RGGs. Heterogeneity in the network is provided
via a richer stochastic model that nevertheless preserves much of the simplicity of the RGG model.  GTGs assign to nodes both a location and a weight. The weight may represent a quantity such as
transmission power in a wireless network or influence in a social
network. Edges are placed between two nodes if a symmetric function of
their weights and the distance between them exceeds a certain threshold~\cite{bradonjic-2007-allerton}. 

Structural properties of GTGs, such as connectivity,
clustering coefficient, degree distribution, diameter, existence and
absence of the giant component, chromatic number have been recently analyzed~\cite{bradonjic-2008-im,bradonjic-2007-waw,bradonjic-2010-dmtcs}.
These properties are not merely of theoretical importance,
but also play an important role in applications.
In communication networks, connectivity implies the ability to reach all
parts of the network. In packet routing, diameter gives the minimal
number of hops needed for transmission between two arbitrary nodes. In the case of epidemics, the existence or absence of the giant component
controls whether the epidemic spreads or is contained.  
When treating the node colors as the different radio channels or frequencies,
the chromatic number gives the minimal number of
channels needed so that neighboring radios do not interfere with each
other.

Herein, we consider random walks on GTGs near the  connectivity threshold.
Random walks (or more formally, Markov chains) on large networks have many applications. For example, random walks model  the spread of disease or the dispersion of information \cite{boyd:gossip}. The \emph{mixing time} of a random walk is the expected number of random steps that are required to guarantee that the current distribution is close to the stationary distribution. Mixing times are an essential tool in both theory and practice: for example, see the recent survey of Diaconis \cite{diaconis:mcmc} on  Markov chain Monte Carlo methods. 
In~\cite{bhamidi-2008}, the authors derived the mixing time of exponential random graphs, a model extensively  used in sociology, showing that the mixing time of the \emph{Glauber dynamics} was $\Theta(n^2 \log n)$  for the unimodal \emph{Gibbs distribution}, and exponential for 
the multimodal case. For the definitions of the Glauber dynamics and Gibbs distribution see~\cite{bhamidi-2008}. 
 
Upper bounds on the mixing time for RGG at the connectivity threshold have been well studied. For the 2-dimensional RGG, Avin and Ercal~\cite{avin-2007-cover} showed that  the mixing time is $O(n)$. More recently, Cooper and Frieze~\cite{cooperfrieze:cover} proved that for $d \geq 3$, the mixing time of a $d$-dimensional RGG is $\tilde{O}( n^{2/d})$ (in this notation, the logarithmic factors are suppressed). 
In this paper, we study the mixing times of random walks on $d$-dimensional  GTGs near the connectivity threshold, where $d \geq 2$. We prove that when the node weight distribution decays sufficiently quickly, the mixing time is $\mixbound = \tilde{O}( n^{2/d})$, which matches the mixing bounds for RGG.
This result is formulated more precisely as Theorem \ref{thm:mixing} in the next section.

\section{The GTG model and the mixing time} 

The GTG model is constructed from 
of a set of $n$ nodes placed independently uniformly at random 
into the unit cube in $\R^d$.
A non-negative weight $w_i$, taken randomly and independently from a
continuous probability distribution function $f(w): \R^+ \to \R^+$, is
assigned to each node $v_i$ for $i \in \{1,2,\dots,n\}$. 
For nodes $i,j$ at distance $r_{ij}$, the edge $\{i,j\}$ exists if and only if the
following connectivity relation is satisfied:
\begin{equation}\label{eq:connectivity_relation} 
G(w_i,w_j)D(r_{ij}) \geq \theta_n\,.
\end{equation} 
Here $G(w_i,w_j)$ is the interaction strength between nodes, $D(r)$ is a decreasing function of $r$ and  $\th_n$ is a given threshold parameter that
depends on the size $n$ of the network. 
The interaction
strength $G(w_i,w_j)$
is usually taken to be symmetric and either
multiplicatively or additively separable, i.e., in the form of
$G(w_i,w_j) = g(w_i)g(w_j)$ or $G(w_i,w_j) = g(w_i) + g(w_j)$. 
We use $D(r) =r^{-s}$ where $s>0$, which is a typical attenuation in 
the path-loss model in wireless communications~\cite{bradonjic-2007-allerton}.

Some basic results have already been shown, including
the expected degree of a node with given weight $w$,
when the nodes are distributed uniformly over a unit space~\cite{masuda-2005-geographical,bradonjic-2007-allerton}.
In both the multiplicative and additive cases of $G(w_i,w_j)$, questions
of diameter, connectivity, and topology control have been
addressed~\cite{bradonjic-2007-allerton}. 

Here we restrict ourselves to nodes distributed uniformly over $[0,1]^d$. For analytical
simplicity we take the space to be the $d$-dimensional toric unit cube $[0,1]^d$. Our connectivity relation uses an additive interaction strength  $G(w_1, w_2) = w_1 + w_2$ and a decay function $D(r)=r^d$, so that nodes $i,j$ are adjacent when
\begin{equation}\label{eq:connectivity_r2}
\frac{w_i + w_j}{r_{ij}^d} \geq \theta_n.
\end{equation}
This connectivity relation identifies a $d$-dimensional sphere of influence for each vertex.

We assume that our weight distribution $f(w)$ has finite mean and finite variance. Let the cumulative density function (cdf), for the distribution of node weights $f(w)$, be
\begin{equation}
\label{eq:cdf}
F(x) = \pr[W \leq x] = \int_0^x f(w) dw.
\end{equation}
The argument in \cite{bradonjic-2007-waw} characterizing the degrees of a GTG for a $2$-dimensional GTG is easily generalized to dimension $d$ (only the leading constant changes).  
For any threshold $\theta_n = O(n)$ and any weight distribution such that 
$\pr[W \geq x] = O(1/x^{1+\eps})$ for an arbitrarily small constant $\eps>0$, 
the degree distribution of a node $v$ with weight $w$ follows the binomial distribution
\begin{equation}
\label{eq:degreedist}
\deg(v | w) \sim \Bin \left(n-1, p(w) \right), 
\end{equation}
where $p(w) = \frac{\Upsilon_d}{\theta_n}(w + \mu)$, 
$\mu=\E[W]$ is the expected node weight and
 $\Upsilon_d$ is the volume of the unit ball in $d$ dimensions:
\begin{equation}
\nonumber
\Upsilon_d = \frac{\pi^{d/2}}{ \Gamma(d/2 + 1)} = 
\left\{
\begin{array}{ll}
\pi^k/k!, & d =2k \mbox{ even}, \\
2^d k! \pi^k / d!, & d=2k+1 \mbox{ odd}.
\end{array}
\right.
\end{equation}
Herein, we assume that  $\pr[W \geq x] = O(1/x^{d+\nu}) = o(1/x^2)$ which  ensures that the weight distribution has finite mean and finite variance.

We now highlight the differences between the GTG and RGG models.
The main characteristic of the GTG model  is its tunable topology. By changing the input parameters $f(w)$ and $\theta_n$, one can obtain graphs with different structural properties. For example, we can generate an RGG for a desired degree distribution, while the degrees of RGG are always uniform. 
The major distinction between our analysis for GTGs and the analysis for RGGs, lies in addressing the following two issues: 
(i) two spatially close nodes in a GTG are not necessarily connected, since they may both have very low weight;  
(ii) two distant nodes in a GTG are not necessarily disconnected, since one of them can carry heavy weight. 
These two issues never happen in the RGG model, and represent a challenge in our analysis. 

Theorem 5.3 of~\cite{bradonjic-2008-im} characterizes the  connectivity threshold
for a GTG in $2$ dimensions. We list the changes to the proof to generalize to $d$-dimensional GTGs. We tile the unit space $[0,1]^d$ into 
$\Theta(n / \log n)$ cubes of equal volume (as opposed to 2-dimensional squares). 
We use the connectivity relation equation (\ref{eq:connectivity_r2}) for the $d$-dimensional space and connectivity radius $\rc = \left( \log (\alpha n) /( \alpha n \Upsilon_d ) \right)^{1/d}$ instead, with the constant $\al \in (0,1)$ specified in Theorem 5.3 of~\cite{bradonjic-2008-im}.
With these modifications, we obtain the proof of the connectivity threshold for a $d$-dimensional GTG, given by the following theorem.
\begin{theorem}
\label{thm:connect}
Let $G$ be a GTG in the $d$-dimensional toric unit cube $[0,1]^d$ with threshold function $\theta_n = c n / \log n$ where the constant 
$c < \sup_{\alpha \in (0,1)} \alpha F^{-1}(1-\alpha)/4$. Then $G$ is connected whp\footnote{
We will use the notation ``with high probability'' and denote it as \textit{whp}, meaning with probability $1-o(1)$ as $n$ tends to infinity.}.
\end{theorem}
In essence, the proof of Theorem \ref{thm:connect} consists of two parts. First, one shows that the $\alpha n$ nodes of highest weight are connected. Next, one shows that the remaining $(1-\alpha)n$ nodes are connected to first set. This partition into \emph{high weight nodes} and \emph{low weight nodes} will also be useful herein.

This paper gives an upper bound on  the mixing time for a simple random walk on a geographical threshold graph $G=(V(G),E(G))$ at the connectivity threshold, provided that the weight distribution decays at an adequate rate.
A simple random walk on a graph $G$  (cf. \cite{lovasz96}) consists of a sequence of vertices $(w_0, w_1, \ldots, w_t, \ldots )$ such that
for $t \geq 0$, $\Pr[ w_{t+1} = j \mid w_t= i ]$ is $1/\deg(i)$ if $\{i,j \} \in E(G)$ and 0 otherwise. 
Let $\pi$ denote the \emph{stationary distribution} of this random walk, so that $\pi_k = \deg(k)/2|E|$ for every node $k$.  Supposing that  $i$ is the initial node of our random walk, let $P^t(i,\cdot)$ denote the distribution of the states at time $t$.  The \emph{variational distance} at time $t$ is
\[
\Delta_i(t) = \frac{1}{2} \sum_{j \in V} \left| P^t(i,j) - \pi_j \right|.
\] 
When $G$ is not bipartite, we have $\lim_{t \rightarrow \infty} \Delta_i(t) = 0$ for every $i \in V(G)$.
The \emph{mixing time} from node $i$ measures how quickly $P^t(i,\cdot)$ converges to $\pi$. 
Explicitly, this mixing time from $i$ is defined as
\[
\tau_i(\delta) =  \min \{ t \mid \Delta_i (t') \leq \delta, \forall t' \geq t \}.
\] 
The mixing time of $G$ is 
\[
\tau(\delta) = \max_{i \in V(G)} \tau_i(\delta).
\]
We choose $\delta = 1/n$ as our desired distance from the stationary distribution. Our main result is given by the following theorem. 
\begin{theorem}
\label{thm:mixing}
Let $G$ be a connected GTG with threshold function $\theta_n = cn / \log n$ in 
the $d$-dimensional toric unit cube $[0,1]^d$.
If the weight distribution satisfies $\pr[W \geq x] = O(1/x^{d+\nu})$ for $\nu >0$,
and 
\begin{equation}
\nonumber
c \leq \frac{1}{d+3} \min \left\{  \sup_{\alpha \in (0,1/2]} \alpha F^{-1} (\alpha), 
\sup_{\alpha \in (1/2,1)} (1-\alpha) F^{-1} (\alpha)  \right\}\,,
\end{equation}
then $\tau(1/n) = \mixbound$, whp.
\end{theorem}
This mixing bound for GTG  matches the best known mixing bound for RGG, and we believe that this equivalence is essentially correct. Intuitively, the $\alpha n$ high weight nodes of a GTG  $G$ contain a spanning subgraph $G'$ that is an RGG. The mixing time of $G$  corresponds to the mixing time of $G'$. There are some extremely long edges in $G \backslash G'$, but they seem to be too sparse to aid in mixing. At the same time, we find that the $(1-\alpha)$ low weight nodes (with very short edges) do not slow down mixing. 
For technical reasons, we consider a weight decay of $\pr[W \geq x] = O(1/x^{d+\nu})$. 
We conjecture that this equivalence continues to hold for  $\pr[W \geq x] = o(1/x^{1+1/d})$.

The rest of the paper is organized as follows. In Section~\ref{sec:node_weights}, we derive upper and lower bounds on the maximal weight. Consequently, we find upper and lower bounds on the node degrees, and show that the number of edges $|E(G)| = \Theta(n \log n)$.
In Section~\ref{sec:canonical} we construct a family of canonical paths for the GTG and then prove
Theorem \ref{thm:mixing} in Section~\ref{sec:mixing_time_for_GTG}. In Section~\ref{sec:conc},
 we reflect on our results and explain why we believe that our result holds for more slowly decaying weight distributions. Finally, in~\ref{sec:examples},  we exemplify our results with  two different weight functions. Our first example is the exponential weight distribution $f(w) = e^{-w}$, with cumulative density function
$F(x) = 1 - e^{-x}$. Our second example is the Pareto distribution with cumulative density function 
$F(x) = 1 - x^{- \gamma}$, where $x \geq 1$ and $\gamma > d \geq 2$.

\section{Node weights and node degrees in GTG}
\label{sec:node_weights}
In this section, we determine the upper and lower bounds on the maximal weight $W_{\max}$ in a geographical threshold graph $G= (V(G), E(G))$. Subsequently, we derive the upper and lower bounds on the degrees of the nodes in GTG near the connectivity threshold. Finally, we show that the number of edges $|E(G)| = \Theta(n \log n)$ for these connected GTGs.

We adopt the following notation for the remainder of the paper. We have a constant $\alpha \in (0,1)$
and we fix small constants $\epsilon$, $\nu$ so that
\begin{equation}
\label{eq:epnu}
0 < \epsilon < \nu/2d.
\end{equation}
Furthermore, we assume that there is a weight $W_0$ such that if $W \geq W_0$ then $\pr[W \geq x] = O(1/x^{d+\nu})$. For brevity, we will state this as ``$\pr[W \geq x] = O(1/x^{d+\nu})$.''
Finally, we
use  $\omn$ to denote an arbitrarily slowly increasing function of $n$, that is $\omn = \omega(1)$.

The maximal weight $W_{\max}$ satisfies $\pr[W_{\max} \leq x] = F(x)^n$, since the weights are independently distributed.
Consider a continuous weight distribution $f(w)$ with cdf $F(x)$. Our goal is to find two thresholds $W_1, W_2$, such that $\pr[W_{\max}  \leq W_1] = o(1)$ and $\pr[W_{\max} \geq W_2] = o(1)$.
We can `invert' $F(x)$ using the quantile function
$F^{-1}(p)= \inf \{ x \in \R^+ : p \leq F(x) \}$. 
Define
\begin{eqnarray*}
W_1 = F^{-1}\left(1 - \frac{\omn}{n} \right) 
& \mbox{and} &  
W_2 = F^{-1}\left(1 - \frac{1}{n \, \omn} \right).
\end{eqnarray*}
We have
$\pr[W_{\max} \leq W_1] = (1-\omn/n)^n \leq \exp(- \omn) \rightarrow 0$ and
$\pr[W_{\max} \leq W_2] = (1- 1/n \,  \omn)^n \geq \exp(- 1/ \omn + 1/n \, \omn^2) \rightarrow 1$.
In conclusion,  the maximal weight $W_{\max}$ satisfies 
\begin{equation}
\label{eq:maxweight}
\lim_{n \to +\infty} \pr[W_{\max} \in(W_1,W_2)] = 1.
\end{equation}
See~\ref{sec:examples} for concrete examples of the calculation of the bounds on the maximal weight.

Let us determine the upper and lower bounds for the node degrees, keeping in mind that the weight distribution has finite mean and variance.
We consider the GTG around the connectivity regime, as described in Theorem \ref{thm:connect}. 
The next result generalizes Lemma 3 in \cite{cooperfrieze:cover}, which shows that all degrees of RGGs near the connectivity threshold are $\Theta(\log n)$ whp.

\begin{lemma}
\label{thm:gtginterval}
Let $G$ be a connected GTG with threshold function $\theta_n = c n / \log n$.
Whp, the nodes $v \in V(G)$ satisfy $\deg(v) \in I_{GTG}$ where
\begin{equation}
\label{eqn:interval}
I_{GTG} = \left[  c_1  \log n, c_2  \,  F^{-1}\left(1- \frac{1}{n \, \omn}\right)   \log n \right],
\end{equation}
 for any function $\omn = \omega(1)$ and for constants 
$$
c_1 =  \frac{ \mu \Upsilon_d}{ c }   \left(1 -  \sqrt{\frac{2c}{ \mu \Upsilon_d}}\, \right) \mbox{ and }  
c_2 = \frac{2 \Upsilon_d}{c}.
$$
\end{lemma}
The minimum degree of a GTG is $\Omega(\log n)$. The maximum degree depends upon the decay rate of the weight distribution: a slower decay rate results in larger maximum degree.
In  \ref{sec:examples}, we calculate $I_{GTG}$ for weight distributions with exponential decay and polynomial decay.

\begin{proof}
Using equation \eqref{eq:degreedist},
we apply the Chernoff bound on the degree $\deg(v|w)$ of a node $v$ with a given weight $w$:
\[
\pr \Big[\deg(v|w) \leq (1 - \delta)  \E[\deg(v|w)] \Big] \leq \exp( - \E[\deg(v|w)] \delta^2/2).
\]
Let $c_3 = 1-  \sqrt{\frac{2c}{ \mu \Upsilon_d}} < 1$.
By having $\E[\deg(v|w)] = (n-1)p(w) = \frac{\Upsilon_d(n-1)}{cn} (w + \mu) \log n $ 
and choosing 
$\dl = 1 - c_3/(1 + w/\mu)$, 
it follows
\begin{equation}
\label{eq:bound_deg_distrib_exp}
\pr \Big[ \deg(v|w) \leq c_1 \log n \Big]  \leq \exp \left(  - \frac{\mu \Upsilon_d}{c} \left( 1 + w/\mu \right) \left( 1 - \frac{c_3}{ 1 + w/\mu} \right)^2 \left(1 - \frac{1}{n} \right)  \log n \right).
\end{equation}

Next, we specify conditions such that  equation (\ref{eq:bound_deg_distrib_exp}) is $o(1/n)$ for all $w \geq 0$ and sufficiently large $n$. For the sake of simplicity, let us denote $x= 1 + w/\mu \geq 1$,  and consider the function 
$$\phi(x) = \frac{\mu \Upsilon_d}{c} x \left(1 - \frac{c_3}{x} \right)^2.$$ 
The minimum of $\phi(x)$ is attained at $x = c_3$. Moreover, $\phi(x)$ is strictly decreasing on $(0,c_3)$ and strictly increasing on $(c_3, + \infty)$. Because $c_3  < 1$ and $\phi(1) = \frac{\Upsilon_d \mu}{c}(1 - c_3)^2 = 2$, we know that
 $\phi(x) \geq 2$ for $x \geq 1$. That is, equation (\ref{eq:bound_deg_distrib_exp}) is $o(1/n)$, for  $n \geq 3$. 
Thus, the degree distribution satisfies 
\begin{eqnarray*}
 \pr \Big[\deg(v) \leq c_1 \log n\Big] &=& \int  f(w) \pr \Big[\deg(v|w) \leq c_1 \log n\Big] dw \\
&=& o \Big( n^{-1} \int  f(w) dw \Big) = o(1/n).
\end{eqnarray*}
The union bound gives the lower bound on degree of the nodes in the graph.

We now obtain the upper bound. Equation (\ref{eq:maxweight}) ensures 
$F^{-1} \left(1 - \omn/n \right) \leq W_{\max} \leq F^{-1}\left(1 - 1/ n \omn \right).$ 
Moreover, by the continuity of $F^{-1}(x),$  for any $\epsilon>0$, there is sufficiently large $n=n(\epsilon)$, such that the upper and lower bounds on $W_{\max}$ are arbitrarily close $0 \leq W_2 - W_1 \leq \epsilon$.

The degree of the node with maximal weight satisfies the binomial distribution $\Bin(n-1, (\Upsilon_d/\theta_n)(W_{\max}  + \mu))$, which is concentrated around its mean $(\Upsilon_d/c) (1 - 1/n) (W_{\max} + \mu) \log n $. Finally, the union bound gives the upper bound on the degrees.
\end{proof}

We now partition the interval $I_{GTG}$ of equation \eqref{eqn:interval}. We use this partition to calculate the number of edges $|E(G)|$ and again in Section~ \ref{sec:mixing_time_for_GTG} to bound the mixing time. 

Define $h(x) = (1 - F(x))^{-1}$, or equivalently,  $F(x) = 1 - 1/h(x)$. This is also equivalent to
$$h^{-1}(y) = F^{-1}(1 - 1/y).$$
By assumption we have  $h(x) = \Omega(x^{d + \nu})$, so that $h^{-1}(x) = O \left(x^{1/(d + \nu)} \right)$. Our first interval $B$ contains the low weight nodes: 
\begin{equation}
\label{eq:defB}
B = \{ v \in V(G) \mid w_v \leq F^{-1}(1-\alpha) \} = \{ v \in V(G) \mid F(w_v) \leq 1 - \alpha \}.
\end{equation}
Next, we partition the $\alpha n$ nodes with weights in $[F^{-1}(1-\alpha), W_{\max}]$.
By equation \eqref{eq:maxweight}, $W_{\max} = F^{-1}\left(1 - 1/n\omn\right) = h^{-1}(n \omn)$.  Let
\begin{eqnarray*}
a_0 &= & W_{\max} \\
a_k &=& \max \left\{ F^{-1} \left( 1 - a_{k-1}^{-(1+\epsilon)} \right), F^{-1}(1-\alpha) \right\}, \textrm{ for } k \geq 1.
\end{eqnarray*}
The $a_k$ are only defined until we reach $F^{-1}(1-\alpha)$. Call this final index $M$. Our partition consists of the subintervals of the form $(a_k, a_{k-1}]$ for $1 \leq k \leq M$. Note that the indexing of our endpoints is the reverse of the standard convention.
\begin{lemma} 
\label{lemma:M}
The final index satisfies $M = o(\log n)$.
\end{lemma}
\begin{proof}
We have $a_0 = W_{\max} = h^{-1} (n \omn) = O \left( (n \omn)^{1/(d+\nu)} \right)$. Let $\beta = (1+\ep)/(d + \nu) < 1$. By induction, for 
$0 <  k < M$, 
$$a_k = 
F^{-1} \left( 1 - a_{k-1}^{-(1+\epsilon)} \right) =
h^{-1}\left(a_{k-1}^{1+\ep}\right) = O\left( W_{\max}^{\beta^k} \right)
= O\left( (n \omn)^{\beta^k/(d + \nu)} \right).$$
By definition, $a_M = F^{-1}(1-\alpha)$. Let $N$ be the smallest integer such that $ (n \omn)^{\beta^N/(d + \nu)} < F^{-1}(1-\alpha)$. The right hand side is a constant, so
$M = O(N) = O( \log \log (n \omn)) = o (\log n)$.
\end{proof}

For $1\leq k \leq M$, let 
\begin{equation}
\label{eq:defAk}
A_k = \left\{ v \in V(G) \mid w_v \in (a_k, a_{k-1}] \right\}.
\end{equation}
The degree of $v \in A_k$ with weight $w_v$ is 
\begin{equation}
\label{eqn:degree}
\deg(v) = O(  (1 + w_v) \log n) = O(a_{k-1} \log n).
\end{equation}
We now show that the sizes of the $A_k$ are all concentrated around their means.

\begin{lemma}
\label{lemma:size-Ak}
Whp, for all $1 \leq k \leq M$, we simultaneously have
$|A_k| = \Theta \left( n/a_{k-1}^{1+\epsilon} \right).$
\end{lemma}

\begin{proof}
We have  $\Pr[W > a_k] = 1/a_{k-1}^{1+\ep}$. Therefore 
$ \E[|A_1|] = n/a_0^{1+\epsilon}$ and
$ \E[|A_k|] = n/a_{k-1}^{1+\epsilon} - n/a_{k-2}^{1+\epsilon} =
\Theta \left( n/a_{k-1}^{1+\epsilon} \right)$
for $2 \leq k \leq M$.

As for the concentration of these values, we consider the most delicate case of $A_1$.
We have $\Pr[v \in A_1] = \Pr[w_v > a_1] = a_0^{-(1+\ep)} = W_{\max}^{-(1+\ep)}  := q_1$. We can consider  
$A_1$ as being generated from $\Bin(n, q_1)$. The Chebyshev inequality gives
\[
\Pr \big[ \left| |A_1| - nq_1 \right| \geq nq_1 \big] \leq \frac{1}{nq_1} = 
O \left( \frac{( n \omn) ^{(1 + \ep)/(d+\nu)}}{n} \right) = O(n^{-1/2}).
\]
Similarly, the probability that each remaining $|A_k|$ is more than twice its mean is $O(n^{-1/2})$. There are $M=o(\log n)$ such sets, so taking a union bound shows that all of them are concentrated whp.
\end{proof}

See~\ref{sec:examples} for concrete examples of the partition of the nodes according to weight.

\begin{lemma}
\label{lemma:sum}
$\sum_{k=0}^M a_k^{-\epsilon} = \Theta(1)$.
\end{lemma}

\begin{proof}
We use the general d'Alembert's convergence criterion:  The sum of positive terms
$\sum_{k=0}^{\infty} c_k$ is convergent if there exist a positive integer $N$ and $\eta >0$ such that  $k>N$ guarantees $c_{k}/c_{k-1} < 1 - \eta$. 

Let $b_0 = F^{-1}(1-\alpha)$ and $b_k = h(b_{k-1})^{1/(1+\epsilon)}$
 for $1 \leq k \leq M$. We have $b_k \leq a_{M-k-1}$ for $0 \leq k \leq M$. It follows that
 \[
 \sum_{k=0}^M \frac{1}{a_k^{\epsilon}} \leq \frac{1}{a_M^{\epsilon}} + \sum_{k=0}^{\infty} \frac{1}{b_k^{\epsilon}}
 \leq 1 + \sum_{k=0}^{\infty} \frac{1}{b_k^{\epsilon}}.
 \]
and
 \[
\frac{1/b_{k}^{\epsilon}}{1/b_{k-1}^{\epsilon} }
=  \left( \frac{b_{k-1}}{b_k} \right)^{\epsilon} = 
\left( \frac{b_{k-1}}{h(b_{k-1})^{1/(1+\ep)}} \right)^{\epsilon} < 1 - \eta
\]
for large $k$ since $h(x) \gg x^{d + \nu} \gg x^{1+\epsilon}$. Therefore this sum converges to a constant (independent of $n$).
\end{proof}

\begin{lemma}
\label{thm:edges}
If $G$ is a GTG at the connectivity threshold then the number of edges $|E(G)| = \Theta(n \log n)$ whp. 
\end{lemma}

\begin{proof}
Lemma \ref{thm:gtginterval} guarantees that all nodes have degree $\Omega( \log n)$ whp, 
so the number of edges $|E(G)| = \Omega(n \log n)$. As for the upper bound, whp
\begin{eqnarray*}
2|E(G)|
 \!\! & = & \! \sum_{v \in B} \deg(v) + \sum_{k=1}^{M} \sum_{v \in A_k} \! \deg(v) 
\, = \, O \Big( \Big( |B|  + \sum_{k=1}^{M} |A_k|  a_{k-1} \Big) \log n \Big) \\
&=& O \Big(
\Big( \alpha n + \sum_{k=1}^M \frac{n}{a_{k-1}^{1+\ep}}  a_{k-1} \Big) \log n 
\Big) 
\, = \, O(n \log n),
\end{eqnarray*}
by Lemma \ref{lemma:sum}.
\end{proof}

\section{Canonical paths for GTG}
\label{sec:canonical}

We employ canonical paths (as introduced in  \cite{jerrumsinclair}) to calculate our bound on the mixing time. In this section, we construct the canonical paths for $G$,  a connected GTG with threshold function $\theta_n = c n / \log n$.
For every ordered pair of nodes $u,v \in V(G)$ we choose a canonical path $\gamma_{uv}$ between them. We define 
\begin{eqnarray}
\label{eq:rho}
\rho 
&=& \max_{e=\{x,y\} \in E(G)} \frac{1}{\pi(x) P(x,y)} 
\sum_{\gamma_{uv} \ni e} \pi(u) \pi(v) | \gamma_{uv} | \nonumber \\
 &=&
\max_{e=\{x,y\} \in E(G)} \frac{1}{2|E(G)|} 
\sum_{\gamma_{uv} \ni e} \deg(u) \deg(v) | \gamma_{uv} | \,,
\end{eqnarray}
where $| \gamma_{uv} |$ is the length of the canonical path from $u$ to $v$. As per 
\cite{jerrumsinclair} Proposition 12.1, the mixing time from node $i$ satisfies
\begin{equation}
\label{eq:mix}
\tau_i(\delta) \leq \rho \left( \log \pi(i)^{-1} + \log \delta^{-1} \right),
\end{equation}
where $\delta >0$. We will set $\delta = 1/n$, so that
$\tau_i(1/n) \leq \rho \left( \log \pi(i)^{-1} + \log n \right).$

Let $G$ be a GTG at the connectivity threshold $\theta_n = cn /\log n$ for $c  < \sup_{\alpha \in (0,1)} \alpha F^{-1}(1-\alpha)/4$, as per Theorem \ref{thm:connect}. The proof of this result in \cite{bradonjic-2008-im}  establishes the following two facts. First, a constant fraction of the nodes $\alpha n$ have weights greater than $F^{-1}(1 - \alpha)$.
We let $H(G) = \{ i \in V(G) \mid w_i > F^{-1}(1-\alpha) \}$ denote this set of \emph{high weight nodes} and let $L(G) =  \{ i \in V(G) \mid w_i \leq F^{-1}(1-\alpha) \}$ denote the complementary set of
\emph{low weight nodes.} Second, each high weight node is connected to every node within the \emph{critical radius}
\begin{equation}
\label{eq:criticalradius}
\rc = \left( \frac{\log (\alpha n)}{ \alpha n \Upsilon_d } \right)^{1/d}.
\end{equation}
In other words, the induced subgraph on $H(G)$ contains a subgraph $G'$ that is a connected RGG. Every vertex in $L(G)$ is adjacent to nodes in $H(G)$, so $G$ is also connected.

We use the connected RGG subgraph $G'$  to construct our canonical paths.
Our construction for the canonical paths is similar to the one used in  \cite{cooperfrieze:cover} to bound the mixing time of RGG. Compared to that result, our proof addresses a novel technical challenge: all of the degrees of a RGG are  
$\Theta(\log n)$, while the degrees of a GTG are heterogeneous. This leads to two challenges. First, when $u,v \in L(G)$, the intermediate nodes on the canonical path $\gamma_{uv}$ must all be in $H(G)$. Therefore, we must choose these paths so that the intermediate nodes are evenly distributed among the high weight nodes. Second, many high weight nodes have degrees that are $\omega(\log n)$. We must ensure that their contribution to the sum in equation \eqref{eq:rho} does not lead to an increase in the mixing time, compared to a RGG.

We now describe the geometric scaffold for our canonical paths, as in  \cite{cooperfrieze:cover}. Partition the unit cube into a toric grid of $[k]^d$ small cubes, where $k$ is specified below. A set of canonical paths for the grid will act as the framework for our canonical paths for the GTG. 
Intuitively, we increment the entries $(\bmod \, k)$ in succession. So first we increase $a_1$ until we achieve $b_1$, then do the same for the second entry, and so on. 
The canonical path from $(a_1, a_2, \ldots , a_d)$ to 
$(b_1, b_2, \ldots , b_d)$ is
\begin{equation}
\label{eq:canonicalpaths}
\begin{array}{l}
\!\! (a_1, a_2, \ldots , a_d), \; (a_1 +1, a_2, \ldots , a_d), \, \ldots , \; (b_1 -1, a_2, \ldots , a_d), 
\; (b_1, a_2, \ldots , a_d),\\
\hspace{.2in} (b_1, a_2 + 1, \ldots , a_d), \; \ldots , \; (b_1, b_2, \ldots , b_d -1), 
\; (b_1, b_2, \ldots , b_d). \\
\end{array}
\end{equation}
Each path has length at most $d k$. 
Note that we always increment the index by $+1$ (even if there is a shorter path).

While there are $k^{2d}$ canonical paths, each edge appears in no more than $k^{d+1}$ paths. Indeed, any path that includes the edge from $(i_1, \ldots i_t, \ldots i_d)$ to
$(i_1, \ldots i_t+1, \ldots i_d)$ must start at $(\l_1, \ldots \l_t, j_{t+1}, \ldots , j_d)$
and end at $(j_1, \ldots j_t-1, \l'_{t}, \ldots , \l'_d)$ for some $\l_1, \ldots \l_t$ and $\l'_t, \ldots , \l'_d.$
This results in $t$ choices for $\l_1, \ldots , \l_t$ and $d-t+1$ choices for $\l'_t , \ldots l'_d$.

Before constructing the canonical paths for our GTG, we must prove two lemmas.
Tile $[0,1]^d$ into cubes $S_i$ with  side length $1/k=(\const \log n/n)^{1/d} = \Theta(\rc)$, where we state the conditions on constant $\const$ later. We have a $[k]^d$ grid,  whose cubes each have volume  
$\const \log n / n$. Let $H(S_i)$ and $L(S_i)$ denote the high weight and low weight nodes in  $S_i$, respectively.

\begin{lemma}
\label{lemma:square}
There exist constants $\beta_0, \beta_1 > 0$ such that whp every cube $S$ satisfies
$\beta_0 \log n \leq | L(S) |  \leq \beta_1 \log n$ and 
$\beta_0 \log n \leq | H(S) |  \leq \beta_1 \log n.$
\end{lemma}

\begin{proof}
This proof is similar to the proof of Theorem 3 in~\cite{bradonjic-2008-im}, and will use the same notation. 
Let $B_i = |H(S_i)|$ and $R_i = |L(S_i)|$.
In expectation, there are $\E[B_i] = \al \const \log n$ high weight
nodes within $S_i$. 
Using the lower and upper tail Chernoff bounds~\cite{chernoff81}, it follows
\begin{eqnarray*}
\pr\left[B_i \leq (1-\dl_1) \E[B_i] \right] &\leq &\exp \left(-\frac{\dl_1^2}{2} \E[B_i] \right), \\
\pr\left[B_i \geq (1+\dl_2) \E[B_i] \right] &\leq &\exp\left(-\frac{\dl_2^2}{2(1+\dl_2)} \E[B_i] \right).
\end{eqnarray*}
Fixing $\dl_2 \in [0,1]$, we take $\dl_1 = \dl_2/\sqrt{1+\dl_2}$, and thus $\delta_1 \in [0,1]$. 
The number of high weighted nodes $B_i$ within the cube $S_i$ satisfies
\begin{eqnarray*}
\pr \left[ B_i \in \left(1-\dl_1, 1+\dl_2 \right)  \E[B_i]  \right] &\geq& 1 -  2 \exp \left(-\frac{\dl_2^2}{2(1+\dl_2)} \E[B_i] \right)\\
&=& 1 -  2 n ^{-\al \const \dl_2^2 / (1+\dl_2)}.
\end{eqnarray*}
By the union bound it follows:
\begin{eqnarray*}
\pr \left[\bigcap_i \Big\{ B_i \in \left(1-\dl_1, 1+\dl_2 \right)  \E[B_i]  \Big\} \right] &\geq & \Big(1 -  2 n ^{-\al \const \dl_2^2 / (1+\dl_2) }\Big)^{n/(\const \log n)} .
\end{eqnarray*}
Taking the limit of the last expression as $n \to +\infty$, for 
$\const \geq (1+\dl_2)/ \al \dl_2^2,$
we obtain the concentration on $B_i$, for each cube $S_i$:
\begin{equation*}
\lim_{n \to \infty} \pr \left[ \bigcap_i \Big\{ B_i \in \left(1-\dl_1, 1+\dl_2\right) \al \const \log n \Big\} \right]  = 1.
\end{equation*} 
The concentration on the number of low weight nodes $R_i$, within each cube $S_i$, follows analogously to the previous analysis. 
In expectation, there are $\E[R_i] = (1-\al) \const \log n$ low weight nodes within $S_i$. Hence by the Chernoff tail bounds and the union bound we have 
\begin{equation*}
\lim_{n \to \infty} \pr \left[\bigcap_i \Big\{ R_i \in \left(1-\dl_1, 1+\dl_2 \right) (1-\al) \const \log n \Big\}\right]  = 1.
\end{equation*} 
Finally, we can guarantee the concentration of both the high and low weight nodes by taking $\const \geq \max \Big\{ \frac{1+\dl_2}{\al \dl_2^2} , \frac{1+\dl_2}{(1-\al) \dl_2^2} \Big\}$, that is, 
\begin{equation}
\label{eq:c0new} 
\const \geq \left\{ 
\begin{array}{ll} 
\frac{1+\dl_2}{\al \dl_2^2}, & \text{if } \al \in(0,1/2)\\ 
\frac{1+\dl_2}{(1-\al) \dl_2^2}, & \text{if } \al \in [1/2,1).
\end{array} \right.
\end{equation}
So the lemma holds with $\beta_0 = \left( 1 - \dl_2/\sqrt{1 + \dl_2} \right) \min \{ \alpha, 1- \alpha \}$ and $\beta_1 = \left( 1 + \dl_2 \right) \max \{ \alpha, 1- \alpha \}.$
\end{proof}

Two cubes are \emph{adjacent} if they share a $(d-1)$-dimensional boundary.
\begin{lemma}
\label{lemma:cbound}
Let $S_i, S_j$ be adjacent cubes. The number of edges between $H(S_i)$ and $H(S_j)$ is $\Omega( \log^2 n)$ provided that 
$c \leq (d+3)^{-d/2} \min \{ \sup_{x \in (1/2,1)} (1-x) F^{-1} (x),$ $\sup_{x \in (0,1/2]} x F^{-1} (x)  \}$.
\end{lemma}
See~\ref{sec:examples} for explicit calculation of the constant $c$ for two example weight distributions.

\begin{proof}
Consider any high weight node $u \in H(S_i)$ and any high weight node $v \in H(S_j)$, with the weights $w_{u}$ and $w_{v}$, respectively. The distance $r_{u v}$ between $u$ and $v$ is at most $\sqrt{d+3}(\const \log n / n)^{1/d}$. 
Indeed, the furthest points in two adjacent $d$-dimensional unit cubes are at distance $\sqrt{d+3}$ by the $d$-dimensional Pythagorean theorem.

Consider the connectivity relation
\begin{equation*}
\frac{w_{u}+w_{v}}{r^d_{u v}} \geq \frac{F^{-1}(1-\al)+F^{-1}(1-\al)}{(d+3)^{d/2}\const \log n / n} = \frac{2F^{-1}(1-\al)}{(d+3)^{d/2}\const} \frac{ n}{\log n}.
\end{equation*}
High weight nodes $u, v$ are connected with probability one if  $(w_{u}+w_{v})/r^d_{u v} \geq \theta_n = c n / \log n$, which is guaranteed if $2 F^{-1}(1-\al)/((d+3)^{d/2}\const) \geq  c$.
Using equation (\ref{eq:c0new}), we require 
\begin{equation*}
c \leq \frac{2}{(d+3)^{d/2}} \frac{\dl_2^2} {1+\dl_2}  
\left\{ 
\begin{array}{ll} 
\al F^{-1}(1-\al), & \text{if } \al \in(0,1/2)\\ 
(1-\al) F^{-1}(1-\al), & \text{if } \al \in [1/2,1).
\end{array} \right.
\end{equation*} 
Since $\dl_2 \in (0,1)$ is arbitrary, with $\sup_{\dl_2 \in (0,1)} \dl_2^2/(1+\dl_2) = 1/2$, the conditions combine to give
\begin{eqnarray*}
c &\leq& (d+3)^{-d/2} \min \Big\{ \sup_{\al \in (0,1/2)} \al F^{-1} (1-\al), \sup_{\al \in [1/2,1)} (1-\al) F^{-1} (1-\al)  \Big\} \\ 	
 &=& (d+3)^{-d/2} \min \Big\{ \sup_{x \in (1/2,1)} (1-x) F^{-1} (x), \sup_{x \in (0,1/2]} x F^{-1} (x)  \Big\}. 	
\end{eqnarray*}
\end{proof}

We now employ a randomized procedure for choosing canonical paths. This procedure guarantees that no edge appears in more than $\Theta(k^{d+1}) = \Theta(\rc^{-d-1} )= \Theta( (n / \log n)^{(d+1)/d})$ canonical paths. Let us denote cubes as $S(a_1, \ldots, a_d)$ for $1 \leq a_i \leq k$, where 
$(a_1, \ldots, a_d)$ is the location in the $d$-dimensional grid. By Lemma \ref{lemma:square}, we have both $\beta_0 \log n \leq |L(S(a_1, \ldots, a_d))| \leq \beta_1 \log n$ and $\beta_0 \log n \leq |H(S(a_1, \ldots, a_d))| \leq \beta_1 \log n$ for some constants $\beta_0,\beta_1 > 0$.
For each cube $S = S(a_1, \ldots, a_d)$,  evenly partition  $L(S)$ into sets $L_i(S)$, for $1 \leq i \leq \beta_0 \log n$. Each set has size at most $\lceil \beta_1/ \beta_0 \rceil$ whp. Assign each set of low weight nodes to a distinct high weight node  $ h_i \in H(S)$, for $1 \leq  i \leq \beta_0 \log n$. The high weight node $h_i$ is the \emph{high weight representative} for the nodes in 
$L_i(S)$. The key outcome of this assignment is  that each $h \in H(S)$ represents a constant number of  low weight nodes whp.

Consider any ordered pair of nodes  $(x,y)$. Let $x \in S(a_1, \ldots, a_d)$ and $y \in S(b_1, \ldots, b_d)$. We choose a canonical path from $x$ to $y$ as follows.
We use the toric grid to identify the sequence of cubes in the canonical path.
Taking equation \eqref{eq:canonicalpaths} as our framework, we consider cubes
$S(a_1, a_2, \ldots, a_d)$, $S(a_1+1, a_2,  \ldots , a_d)$, $\ldots$, $S(b_1, a_2, \ldots, a_d)$, $S(b_1, a_2 + 1,  \ldots, a_d)$, $\ldots$ ,  $S(b_1, b_2, \ldots, b_d)$. For brevity, call these cubes $S_0, S_1, \ldots, S_t$.
If $x$ is a low weight node, then set $x_0$ to be the high weight representative for $x$. If $x$ is a high weight node, set $x_0=x$. 
For $1 \leq i \leq t$, choose $x_{i}$ to be a random high weight node in $S_i$.
Set $x_t=y$ if $y$  is a high weight node, otherwise using the high weight representative for $y$.
Our canonical $(x,y)$ path is $x, x_0, x_1, \ldots, x_t, y$. If $x_0=x$ or $x_t=y$, we remove the repeated node from the path. We have
\begin{equation}
\label{eq:pathlength}
\max_{u,v \in V(G)} |\gamma_{uv} | \leq d \, k + 2 = O \left( \left(n/\log n\right)^{1/d} \right)
\end{equation}
and furthermore, we can bound how often each edge appears in a canonical path.
\begin{lemma}
Every edge in $G$ appears in at most $O( (n/ \log n)^{(d+1)/d})$ canonical paths whp.
\end{lemma}

\begin{proof}
Let $Z_{xy}$ denote the number of times the edge $xy$ is chosen. 
If both $x$ and $y$ are low weight nodes then $Z_{xy}=0$  (we always move to a high weight node from a low weight node). When $x$ is low weight and 
$y$ is high weight, then the edge $xy$ can only be used when $x$ and $y$ are in the same cube. In this case,  the edge $xy$ can only be chosen when the canonical path has $x$ as one of its endpoints. Therefore $Z_{xy} \leq 2(n-1)$.

Now suppose that both $x$ and $y$ are high weight nodes. The edge $xy$ is only used if $x,y$ are in the same cube or in adjacent cubes. First, we consider high weight $x,y$  in the same cube $S$. 
The edge $xy$ will be used only by paths between $L_i(S) \cup \{ x \}$ and $L_j(S) \cup \{ y \}$. If one or both is not a high weight representative, then this edge will be used even fewer times. Therefore $Z_{xy} \leq (1 + \lceil \beta_1/\beta_0 \rceil)^2 = O(1)$ whp.

Consider $x \in S$ and $y \in T$ in adjacent cubes $S, T$. Let
${\cal P}_{xy}= \{ \gamma_{uv} \mid xy \in \gamma_{uv} \}.$
We consider four cases, according to the locations of $u$ and $v$. If $u \in S$ and $v \in T$ then  just as in the previous case, whp there are at most 
$(1 + \lceil \beta_1/\beta_0 \rceil)^2 = O(1)$ such paths.

The total number of paths $\gamma_{uv}$ with $u \in S$ and $v \notin T$ is $O(n)$. Indeed, $x$ must be the high weight representative of $u$, which gives $1 + \lceil \beta_1/\beta_0 \rceil = O(1)$ choices for $u$, and there are $O(n)$ end nodes for paths that start in $S_1$. Similarly, the total number of paths with $u \notin S$ and $v \in T$ is $O(n)$.

The remaining case is when $u \notin S$ and $v \notin S$.  
Let $Z'_{xy}$ denote the number of times the edge $xy$ is chosen as the random edge from $S$ to $T$ for some canonical path. Let 
\[
{\cal P}(S, T) = \{ \gamma_{uv} \mid \gamma_{uv}\textrm{ traverses from $S$ to $T$ and
$ u \notin S$ and} v \notin T \}.
\]
Then whp
\[
\E[Z'_{xy}] =  \!\!\!\! \sum_{\gamma_{uv} \in  {\cal P}(S, T)} \!\!\!\!\!\!\!\! \pr(xy \in \gamma_{uv}) 
\leq \rc^{-d-1} (\beta_1 \log n )^2   \frac{1}{( \beta_0 \log n)^2} 
=\left( \frac{\beta_1}{\beta_0} \right)^2 \rc^{-d-1}.
\]
Indeed, the number of canonical paths for the toric  grid which pass from $S$ to $T$ is $O(\rc^{-d-1})$. Each edge between cubes  corresponds to at most $(\beta_1 \log n)^2$ toric paths in the GTG. Since both $x$ and $y$ are internal nodes of these paths, they were each chosen uniformly and independently with probability at most $1/ (\beta_0 \log n)$.

Using the Chernoff bound for this binomial distribution, 
\begin{eqnarray*}
\lefteqn{
\pr\left[ | Z'_{xy} - \E[Z'_{xy}] | \geq \eta \E[Z'_{xy}] \right] 
\, \leq \, 2 \exp \Big( -\frac{\eta^2}{ 3} \E[Z'_{xy}] \Big) }  \\
& \leq  & 
2 \exp \Big( -\frac{\eta^2}{ 3} \big(\frac{\beta_1}{\beta_0} \big)^2 \rc^{-d-1} \Big) \\
& = & 2 \exp \Big( -\frac{\eta^2}{ 3} \big(\frac{\beta_1}{\beta_0} \big)^2 
\big( \frac{n}{ \alpha \Upsilon_d \log n} \big)^{(d+1)/d} \Big) 
\, = \, o(1/e^{n}).
\end{eqnarray*}
The union bound now gives
\begin{eqnarray*}
&& \pr \left[ \wedge_{x,y}\left(  |Z'_{xy}- \E[Z'_{xy}] | \geq \eta \E[Z'_{xy}] \right) \right] \\
&\leq& \sum_{x,y} \pr  \left[  |Z'_{xy}- \E[Z'_{xy}] | \geq \eta \E[Z'_{xy}]  \right] \,\leq \,  n^2 o(1/e^{n}) \rightarrow 0.
\end{eqnarray*}
Therefore whp, every edge between high weight nodes in adjacent cubes is used by $(1 \pm \eta) \rc^{-d-1} = \Theta((n/ \log n)^{(d+1)/d})$ canonical paths. In this case, 
$\E[Z_{xy}] = \Theta((n/\log n)^{(d+1)/d})$.
\end{proof}

\section{The mixing time for GTG}
\label{sec:mixing_time_for_GTG}

In this section, we  prove that the mixing time for a $d$-dimensional GTG near the threshold for connectivity is $\mixbound$ when $\pr[W \geq x] = O(1/x^{d+\nu})$ for $\nu > 0$.
 
We have  $|E(G)| = \Theta( n \log n)$  and
$\max_{u,v \in V(G)} |\gamma_{uv}| = O( (n/\log n)^{1/d})$ by Lemma \ref{thm:edges} and  equation~\eqref{eq:pathlength}, respectively.  
Substituting these values into equation~(\ref{eq:rho}) yields
\begin{equation}
\label{eqn:canon-rho}
\rho \leq  
O \left(  \frac{1}{n^{(d-1)/d}(\log n)^{(d+1)/d}} \right) 
\left( \max_{e=\{x,y\} \in E(G)} \sum_{\gamma_{uv} \ni e} \deg(u) \deg(v) \right).
\end{equation}
Fix an edge $e=\{x,y\}$ between high weight nodes in adjacent cubes, and define
$\sigma(e) = \sum_{\gamma_{uv} \ni e} \deg(u) \deg(v).$
Suppose that these high weight cubes differ in the $k$th coordinate. Specifically, the cubes containing $x$ and $y$ are indexed by
$(j_1, \ldots, j_{k-1}, j_k, j_{k+1}, \ldots, j_d)$ and
$(j_1, \ldots, j_{k-1}, j_k+1, j_{k+1}, \ldots, j_d)$, respectively.
If $e \in \gamma_{uv}$ then $u$ must be in an initial cube indexed by
$(i_1, \ldots, i_{k-1}, i_k, j_{k+1}, \ldots, j_d)$, while $v$ must be in a target cube indexed by
$(j_1, \ldots, j_{k-1}, \ell_k, \ell_{k+1}, \ldots, \ell_d)$. 
Let $U_1 \subset [0,1]^d$ correspond to the set of initial cubes, and let $U_2 \subset [0,1]^d$ correspond to the set of target cubes.
The volumes of these subsets are $\vol(U_1) = (n / \log n )^{k/d} / ( n / \log n ) =
(\log n / n)^{(d-k)/d}$ and $\vol(U_2) = (n / \log n )^{(d-k+1)/d} / ( n / \log n ) =
(\log n / n)^{(k-1)/d}$. Let $V_1$ (respectively $V_2$) be the set of nodes in $U_1$ (respectively $U_2$).

In order to bound $\sigma(e)$ we use the partition of the vertices into the sets $B$ and  $A_1, A_2, \ldots A_M$ as defined in equations \eqref{eq:defB} and \eqref{eq:defAk}. We consider canonical paths containing $e$ that run between every possible pair of sets in this partition. The following three technical lemmas require the weight decay of $\pr[W \geq x] = O(1/x^{d+\nu})$.

\begin{lemma}
\label{lemma:T-conc}
For any $T \in \{ B, A_1, A_2, \ldots A_M \}$, let $Z_i  = | T \cap V_i |$ for $i=1,2$. Then 
$\E[Z_i] = |T|  \vol(U_i)$ and
there exists a constant $C >0$ (depending on $T$ and $i$) such that
\[
\Pr \left[ |Z_i - \E[Z_i] | \geq   \E[Z_i]  \right]  = O\left( \exp( -C n^{\nu/2d(d+\nu)}) \right).
\]
\end{lemma}
In other words, the probability that $Z_i \neq \Theta(\E[Z_i])$ decays at a super-polynomial rate.

\begin{proof}
We prove this concentration for $Z = A_M \cap V_1$, which is the most delicate case. The other cases follow similarly.
First, by the independence of weights and location, we have $|A_M \cap V_1| \sim \Bin (|A_M| , \vol(U_1))$.
Furthermore, by Lemma \ref{lemma:size-Ak}, $|A_M| = \Theta(n/a_{M-1}^{1+\ep})$.
Therefore $A_M \cap V_1 \sim \Bin (C' \, n/a_{M-1}^{1+\ep}, (\log n / n)^{(d-k)/d})$ for some $C'>0$.
The expected value is
\begin{eqnarray}
\nonumber
\E[Z]  &=& \frac{C' \, n}{a_{M-1}^{1+\ep}}  \left( \frac{\log n}{n} \right)^{(d-k)/d} \, \geq \,
\frac{C' \, n}{W_{\max}^{1+\ep}}  \left( \frac{\log n}{n} \right)^{(d-k)/d} \\
 \nonumber
 & \geq & \frac{C' \, n}{(n \omn)^{(1+\ep)/(d+\nu)}}   \frac{(\log n)^{1/d}}{n^{(d-1)/d}}   \\
\label{eqn:Z-lower}
&\geq& C' n^{1/d - (1+\ep)/(\delta + \nu)} \, = \, \Omega \left(n^{\nu/2d(d+\nu)} \right)
\end{eqnarray}
since $\ep < \nu/2d$ by equation \eqref{eq:epnu} and $\omn$ grows arbitrarily slowly.
Using the binomial Chernoff bound (cf \cite{janson-2000-random}, Corollary 2.3), we have
\begin{eqnarray*}
{\Pr \left[ \left| Z - \E[ Z ] \right| \geq \ \E[ Z ]  \right]
\leq 2 \exp \left( - \frac{1}{3} \E[ Z ] \right) } = O \left(\exp(- C n^{\nu/2d(d+\nu)}) \right)
\end{eqnarray*}
for some constant $C >0$. 
\end{proof}

For $T_1, T_2 \in \{B, A_1, A_2, \ldots A_M \}$, let $\lambda(T_1, T_2)$ denote the number of canonical paths from $T_1 \cap V_1$ to $T_2 \cap V_2$ that use edge $e$. 

\begin{lemma}
\label{lemma:lambda-conc}
For  $T_1,T_2 \in \{ B, A_1, A_2, \ldots A_M \}$, we have 
$$\E[\lambda(T_1, T_2) ] = \Theta \left( \frac{\E[|T_1 \cap V_1|]  \cdot  \E[|T_2 \cap V_2|] }{ \log^2 n} \right)$$
 and there exists a constant $K$ (depending on $T_1, T_2$) such that
\[
\Pr \left[ \lambda(T_1,T_2) - \E[\lambda(T_1,T_2)] | \geq   \E[\lambda(T_1,T_2)]  \right]  
= O  \left(\exp(- K n^{\nu/2d(d+\nu)})  \right).
\]
\end{lemma}

\begin{proof}
The proof is similar to the previous one. For $i=1,2$, let $Z_i = T_i \cap V_i$, so that  
$|Z_i| = \Theta(\E[Z_i])$  by Lemma \ref{lemma:T-conc}.
 The distribution for $\lambda(T_1, T_2)$ is $\Bin(|Z_1|  |Z_2|, K'/\log^2 n)$ for some $K' >0$,
and $|Z_i| = \Theta(\E[Z_i])$ for $i=1,2$ by Lemma \ref{lemma:T-conc}.

Furthermore,  $|Z_i| = \Omega(n^{\nu/2d(d+v)})$ by equation (\ref{eqn:Z-lower}). Therefore $\E[\lambda(T_1, T_2)] = \Omega (n^{\nu/d(d+\nu)}/\log^2 n) = \omega(n^{\nu/2d(d+\nu)})$.
The binomial Chernoff bound gives
\begin{eqnarray*}
\Pr \big[ \left| \lambda(T_1,T_2) - \E[ \lambda(T_1, T_2) ] \right| \geq \ \E[ \lambda(T_1, T_2) ]  \big]
&\leq& 2 \exp \left( - \frac{1}{3} \E[ \lambda(T_1, T_2) ] \right)  \\
&=& O \left(\exp(- K n^{\nu/2d(d+\nu)}) \right)
\end{eqnarray*}
for some constant $K >0$. 
\end{proof}

The previous two lemmas show that these quantities are tightly concentrated around their means. A union bound shows that whp these quantities are concentrated for all candidate edges simultaneously.
Indeed, there are $O (d \cdot (n /\log n)^{1/d})$ choices for adjacent cubes used in canonical paths, with $\Theta(\log^2 n)$ edges running between each pair. Recall that $M=o(\log n)$ by Lemma \ref{lemma:M}. For a fixed choice of adjacent cubes, there are $2(M+1) = o(\log n)$ choices for $T$ in Lemma \ref{lemma:T-conc}. Our union bound for these event involves $O( (n/\log n)^{1/d}   \log n)$ terms, each decaying at a super-polynomial rate. Therefore, all these events are concentrated whp.
As for Lemma \ref{lemma:lambda-conc}, there are $(M+1)^2 = o(\log^2 n)$ choices for $(T_1,T_2)$ for each pair of adjacent cubes. This union bound is taken over all relevant edges between high adjacent cubes. The number of addends in this union bound is $O((n /\log n)^{1/d} (\log^2 n) (\log^2 n))$. Again, the concentrations from Lemma \ref{lemma:lambda-conc} are super-polynomial, so the union bound shows that all these quantities are concentrated simultaneously whp.

\begin{lemma}
\label{lemma:rho}
Our choice of canonical paths gives $\rho = O \left( (\log n)^{2/d} \right)$.
\end{lemma}

\begin{proof}
We  bound $\sigma(e) =  \sum_{\gamma_{uv} \ni e} \deg(u) \deg(v)$. First consider the contributions from canonical paths from $B \cap V_1$ to $B \cap V_2$.  Let $\sigma(B,B)$ denote the contribution of paths between low weight nodes to $\sigma(e)$.

Recall that if  $v \in B$ then $\deg(v) = \Theta(\log n)$. By Lemmas \ref{lemma:lambda-conc}, and \ref{lemma:T-conc}, we have whp
\begin{eqnarray*}
\E[ \sigma(B, B)]
&=& 
O \left(  \frac{\lambda(B,B)}{\log^2 n} \cdot   (\log^2 n) \right) 
\, = \,
O \left( |B \cap V_1| \cdot  |B \cap V_2 | \right) \\
&= &
O \left(
n \left(\frac{\log n}{ n}\right)^{(d-k)/d} \cdot n \left(\frac{\log n}{ n} \right)^{(k-1)/d}
\right) \\
& = &
O \left( n^{(d+1)/d}(\log n)^{(d-1)/d}\right).
\end{eqnarray*}

We consider the contribution of  paths between high weight nodes by using the partition $A_1, A_2, \ldots, A_M$ specified in equation \eqref{eq:defAk}. Let $\sigma(A_r,A_s)$ denote the contribution to $\sigma(e)$ for paths from $A_r$ to $A_s$.
Recall that, as per equation \eqref{eqn:degree}, if $v \in A_k$ then
$\deg(v) = O(a_{k-1} \log n)$.
Arguing similarly to the calculation above, 
the contribution to $\sigma(e)$ from paths between nodes in $A_r$ and $A_s$, where $0 \leq r,s \leq M$, is whp
\begin{eqnarray*}
\sigma(A_r, A_s)
&=& 
O \left(  \frac{\lambda(A_r,A_s)}{\log^2 n} \cdot a_r a_s \log^2 n \right) \\
&=&
O \left(
\frac{n}{a_r^{1+\ep}}  \left(\frac{\log n}{n}\right)^{(d-k)/d}    
\cdot    \frac{n}{a_s^{1+\ep}}   \left(\frac{\log n}{n}\right)^{(k-1)/d}  a_r a_s  
\right) \\ 
&= &
O \left(\frac{ n^{(d+1)/d}(\log n)^{(d-1)/d}}{a_r^{\ep} a_s^{\ep}} \right).
\end{eqnarray*}
Next we consider the  paths between low weight and high weight nodes:  whp
\begin{eqnarray*}
\sigma(B, A_s)
&=& 
O \left(  \frac{\lambda(B,A_s)}{\log^2 n}  \cdot  a_s \log^2 n \right) \\
&=&
O \left(
{n}  \left(\frac{\log n}{n}\right)^{(d-k)/d}    
\cdot   \frac{n}{a_s^{1+\ep}}   \left(\frac{\log n}{n}\right)^{(k-1)/d}   a_s  
\right) \\ 
&= &
O \left(\frac{ n^{(d+1)/d}(\log n)^{(d-1)/d}}{a_s^{\ep}} \right). 
\end{eqnarray*}
and similarly, $\sigma(A_r, B) = 
O \left( n^{(d+1)/d}(\log n)^{(d-1)/d} /a_r^{\epsilon} \right).$
Putting these estimates together, whp every edge $e$ between cubes satisfies
\begin{eqnarray*}
\sigma(e) &=& \sigma(B,B) + \sum_{j=1}^M (\sigma(B,A_j) + \sigma(A_j,B))
+ \sum_{j=1}^M \sum_{k=1}^M \sigma(A_j, A_k) \\
&=&
 O \left(
n^{(d+1)/d}(\log n)^{(d-1)/d} 
\Big( 1 + 2 \sum_{i=1}^{M} \frac{1}{ a_i^{\ep}} +
 \sum_{j=1}^{M} \sum_{k=1}^{M} \frac{1}{ a_j^{\ep}  a_k^{\ep}} \Big) 
\right) \\
&=&
 O \left(
n^{(d+1)/d}(\log n)^{(d-1)/d} 
  \Big( 1 +  2\sum_{i=1}^{M} \frac{1}{ a_i^{\epsilon}} +
 \Big( \sum_{j=1}^{M} \frac{1}{ a_j^{\epsilon}} \Big)^2 \Big)
\right) \\
&=&
 O \left( n^{(d+1)/d}(\log n)^{(d-1)/d} 
\right)
\end{eqnarray*}
where the last equality follows from Lemma \ref{lemma:sum}.

Finally, by equation \eqref{eqn:canon-rho} we have whp 
$
\rho = O \left( \left( n/\log n \right)^{2/d} \right).
$
\end{proof}

\begin{proofof}{Theorem \ref{thm:mixing}}
Equation (\ref{eq:mix}) gives $\tau_x(\delta) \leq \rho \left( \log \pi(x)^{-1} + \log \delta^{-1} \right).$
The previous lemma ensures that $\rho=O((n / \log n)^{2/d})$. Meanwhile, we have $\pi(x) = \deg(x) / 2|E(G)| = \Omega( \log n / (n \log n)) = \Omega(1/n)$ by Lemma \ref{thm:gtginterval}  and Lemma \ref{thm:edges}. In summary,
$\tau_x(\delta) = O(n^{2/d} (\log n)^{(d-2)/d})$ for $\delta = 1/n$.
\end{proofof}

\section{Conclusion}

\label{sec:conc}

We have shown that if the weight distribution of a $d$-dimensional GTG satisfies $\pr[W > x] = O(1/x^{d+\nu})$, then its mixing time is $O(n^2/d (\log n)^{(d-2)/d})$. This matches the known bounds for RGG. Our proof uses a spanning subgraph among the $\alpha n$ high weight nodes to create a scaffold for canonical paths. In constructing these paths, our proof treats all high weight nodes identically, ignoring the particularly large reach of the highest weight nodes. We did try  to take advantage of these hub nodes, but  found that they were to sparse to leverage for canonical paths. 

One might wonder whether this is a shortcoming of the method of canonical paths, rather than a reflection on the characteristics of GTG. However, initial investigations using conductance to bound mixing (as in \cite{avin-2007-cover}) suggests the same conclusion. Of course, using conductance introduces its own set of technical challenges due to the heterogeneous nature of the degrees.

For technical reasons, we assumed that the weight distribution decayed as
$\Pr[W \geq x] = 1/x^{d+\nu}$. We believe that the equivalence of mixing for GTG and RGG extends to weight distributions with slower decay. In particular, we conjecture that these mixing time of a GTG with weight decay
$\Pr[W \geq x] = O(1/x^{\gamma})$ matches that of  RGG when
$\gamma \geq 1 + 1/d$, and that GTG mixes faster when 
$1 < \gamma < 1 + 1/d$.

\section*{Acknowledgments}  
The first author was supported in part by NSA Young Investigator Grant H98230-08-1-0064. The second author was supported in part through the Laboratory Directed Research and Development Program, and Center for Nonlinear Studies at Los Alamos National Laboratory.

\bibliographystyle{acm}
\bibliography{cover-new}

\appendix
\section{Characteristics of GTG for example weight distributions}
\label{sec:examples}

We describe the relevant characteristics of GTGs for two different weight distributions: exponential decay and   polynomial decay. 

\subsection{Exponential Weight Distribution}
Our first example is the exponential weight distribution $f(w) = e^{-w}$ with cumulative density function
$F(x) = 1 - e^{-x}$. Inverting the cdf gives $F^{-1}(x) = - \log (1 - x)$. 

We first discuss the weights and degrees of the nodes in the GTG, as described in Section \ref{sec:node_weights}. As per equation \eqref{eq:maxweight}, the maximum weight satisfies
$$W_{\max} \in [\log n -  \omn, \log n +  \omn]$$
 whp.
Lemma \ref{thm:gtginterval} guarantees that whp all the node degrees are in the interval
$$I_{GTG} = \left[  \Theta(  \log n), \Theta(\log^2 n) \right].$$ 

Next, we partition the interval $I_{GTG}$.  By
equation \eqref{eq:defB}, the cutoff for low weight nodes is 
 $F^{-1}(1 - \alpha) = \log \alpha^{-1}$. We partition the high weight nodes into disjoint subsets as specified in equation \eqref{eq:defAk}.
 Note that
$F^{-1}\left(1- x^{-(1+\ep)} \right) =  \log ( x^{1+\ep}) = \Theta(\log x)$. 
 Therefore, ignoring leading constants, the sequence of endpoints
(in descending order) is
\begin{equation*}
\big(\log n, \log \log n, \log \log \log n, \ldots , \log^{(M)} n \big) 
\end{equation*}
where $M \leq \log^*n$,
where the iterative logarithm function $\log^* n$ is  the number of iterations of the log function required to obtain a result  less than 1.

Finally, we calculate upper bound on the constant $c$ required by Lemma
\ref{lemma:cbound} in Section \ref{sec:canonical}. For the exponential distribution, taking $c \leq ((d+3)^{d/2}e)^{-1}$ is sufficient.

\subsection{Pareto Weight Distribution}
We now give a parallel characterization of our  second example: 
a Pareto distribution with cumulative density function 
$F(x) = 1 - x^{- \gamma}$ where $\gamma > d \geq 2$. Inverting this cdf gives 
$F^{-1}(x) = (1-x)^{-1/\gamma}$. 

We consider our results concerning node weights and node degrees.
We have
$$W_{\max} \in \left[\left(\frac{n}{\omn} \right)^{1/\gamma}, 
\left(n \omn\right)^{1/\gamma}\right]$$ 
whp by
equation \eqref{eq:maxweight}. 
By Lemma \ref{thm:gtginterval}, whp all the node degrees are in
$$I_{GTG} = [  \Theta(  \log n), \Theta((n \, \omn)^{1/\gamma} \log n) ].$$
We separate the low weight nodes from the high weight nodes using the weight cutoff
 $F^{-1}(1-\alpha) = \alpha^{-1/\gamma}.$
Next, we partition the high weight nodes as per equation \eqref{eq:defAk}.
Note that $F^{-1}\left(1- x^{-(1+\ep)} \right) =  x^{(1 + \ep)/\gamma}$. It follows that our sequence of endpoints is
\begin{equation}
\nonumber
( (n \, \omn)^{1/\gamma}, (n \, \omn)^{\beta/\gamma}, (n \, \omn)^{\beta^2/\gamma}, \ldots , (n \, \omn)^{\beta^{M}/\gamma} )
\end{equation}
where $\beta = (1+ \ep)/\gamma < 1$. Here $M$ is the smallest integer such that $ (n \omn)^{\beta^M / \gamma}  \leq F^{-1}(1-\alpha) = \alpha^{-1/\gamma}$. The latter requirement is equivalent to the condition 
  $M \geq ( \log  \log (n \omn) - \log \log (\alpha^{-1}) ) / \log(\beta^{-1}).$ 

Finally, we calculate an upper bound on the constant $c$ required by Lemma
\ref{lemma:cbound} in Section \ref{sec:canonical}.
Since  $F^{-1}(x) = (1-x)^{-1/\gamma}$, by Lemma~\ref{lemma:cbound} it follows $c \leq ((d+3)^{d/2}2^{(\gamma-1)/\gamma})^{-1}$.

\end{document}